\newenvironment{E}{\begin{equation}}{\end{equation}}
\def\proof{\noindent{\bf Proof: }}
\def\qed{ \hskip 20pt{\vrule height7pt width6pt depth0pt}\hfil}
\def\forb{{\hbox{forb}}}
\def\forbmax{{\hbox{forbmax}}}

\def\c{{\bf c}}

\def\0{{\bf 0}}
\def\1{{\bf 1}}

\def\Av{{\mathrm{Avoid}}}

\newcommand{\ncols}[1]{\| #1 \|}
\newcommand{\rf}[1]{(\ref{#1})}
\newcommand{\trf}[1]{Theorem~\ref{#1}}
\newcommand{\lrf}[1]{Lemma~\ref{#1}}

\newcommand{\srf}[1]{Section~\ref{#1}}

\documentclass[12pt]{article}
\newtheorem{thm}{Theorem}[section]
\newtheorem{lemma}[thm]{Lemma}
\newtheorem{prop}[thm]{Proposition}
\newtheorem{cor}[thm]{Corollary}

\newtheorem{remark}[thm]{Remark}

\usepackage{amssymb}
\usepackage{amsmath}
\title{ Unavoidable Multicoloured Families of Configurations}
\author{R.P. Anstee\thanks{Research supported in part by
NSERC, some work done while visiting the second author at USC.} 
\\Mathematics Department\\The University of British Columbia\\Vancouver,
B.C. Canada V6T 1Z2\\ \\
Linyuan Lu
\thanks{Research supported in part by NSF
grant DMS 1300547 and ONR grant N00014-13-1-0717.}
 \\Mathematics Department\\The University of South Carolina\\Columbia, SC, USA  \\
\\\mbox{\ }}

\begin{document}
\maketitle
\begin{abstract}
Balogh and Bollob\'as [{\em Combinatorica 25, 2005}] prove that for
any $k$ there is a constant $f(k)$ such that any set system with
at least $f(k)$ sets reduces to a $k$-star, an $k$-costar or an
$k$-chain. They proved $f(k)<(2k)^{2^k}$. Here we improve it to
$f(k)<2^{ck^2}$ for some constant $c>0$.

This is a special case of the following result on the multi-coloured 
forbidden configurations at $2$ colours. 
Let $r$ be given. Then there exists a
constant $c_r$ so that a matrix with entries drawn from
$\{0,1,\ldots ,r-1\}$ with at least $2^{c_rk^2}$ different columns will have a
$k\times k$ submatrix that can have its rows and columns permuted so
that in the resulting matrix will be either $I_k(a,b)$ or $T_k(a,b)$
(for some $a\ne b\in \{0,1,\ldots, r-1\}$), where
$I_k(a,b)$ is the $k\times k$ matrix with $a$'s on the diagonal and $b$'s
else where, $T_k(a,b)$ the $k\times k$ matrix with $a$'s below the
diagonal and $b$'s elsewhere. 
 We also extend to considering the bound on the number of distinct
  columns, given that the number of rows is $m$, when avoiding a
  $t k\times k$ matrix obtained by taking any one of the $k \times
  k$ matrices above and repeating each column $t$ times.  We use
  Ramsey Theory.

\end{abstract}

\section{Introduction}

 We define 
 a matrix to be \emph{simple} if it has no repeated columns. A (0,1)-matrix that is simple is the matrix analogue of a set system (or simple hypergraph)
 thinking of the matrix as the element-set incidence matrix.  
 We generalize to allow more entries in our matrices and define
 an $r$-\emph{matrix} be a matrix whose entries are in $\{0,1,\ldots, r-1\}$. We can think of this as an $r$-coloured matrix. We examine extremal problems and let $\ncols{A}$ denote the number of columns in $A$.

We will use the language of matrices in this paper rather than sets. For two  matrices $F$ and $A$, define $F$ to be a \emph{configuration} in $A$, and write $F \prec A$, if there is a row and column permutation of $F$ which is a submatrix of $A$.  Let ${\cal F}$ denote a finite set of matrices. Let 
$$\Av(m,r,{\cal F})=\left\{A\,:\,A\hbox{ is }m\hbox{-rowed and simple }r\hbox{-matrix}, F\nprec A\hbox{ for }F\in{\cal F}\right\}.$$
Our extremal function of interest is
$$\forb(m,r,{\cal F})=\max_A\{\ncols{A}\,:\,A\in\Av(m,r,{\cal F})\}.$$
We use the simplified notation $\Av(m,{\cal F})$ and $\forb(m,{\cal F})$ for $r=2$.
 Many results of forbidden configurations are cases with $r=2$ (and hence $(0,1)-$matrices) and with $|{\cal F}|=1$. There is a survey \cite{survey} on forbidden configurations.  There are a number of results for general $r$ including a recent general shattering result \cite{FS}  which has references to earlier work. This paper explores some forbidden families with constant or linear bounds. We do not require any $F\in{\cal F}$ to be simple which is quite different from usual forbidden subhypergraph problems.

A lovely result of Balogh and Bollob\'as on set systems can be
restated in terms of  a forbidden family.  
Let $I_{\ell}$ denote the $\ell\times\ell$ identity matrix, $I_{\ell}^c$ denote the (0,1)-complement of $I_{\ell}$  and let $T_{\ell}$ be the $\ell\times\ell$ (upper) triangular matrix with a 1 in position $(i,j)$
if and only if $i\le j$.  As a configuration, the   ${\ell}\times{\ell}$ lower triangular matrix with 1's on the diagonal is the same as $T_{\ell}$.
The result \cite{BaBo} show that after a constant number of distinct columns, one cannot avoid all three configurations $I_{\ell},I_{\ell}^c,T_{\ell}$.

\begin{thm}\cite{BaBo} For any $\ell\geq 2$,
  $\forb(m,\{I_{\ell},I_{\ell}^c,T_{\ell}\})\le (2\ell)^{2^\ell}$. 
\label{BB}\end{thm} 

A slightly worse bound  $\forb(m,\{I_{\ell},I_{\ell}^c,T_{\ell}\}) \le
2^{2^{2\ell}}$ but with a simpler proof is in \cite{BKS}.
We generalize $I_{\ell},I^c_{\ell},T_{\ell}$ into $r$-matrices as
follows. Define the generalized identity matrix
 $I_{\ell}(a,b)$ as the $\ell\times\ell$ $\quad
r$-matrix with $a$'s on the diagonal and $b$'s elsewhere.
Define the generalized triangular matrix $T_{\ell}(a,b)$ as the $\ell\times\ell$ $\quad
r$-matrix with $a$'s below the diagonal and $b$'s elsewhere. Now we
have $I_\ell=I_\ell(1,0)$, $I_\ell^c=I_\ell(0,1)$, and
$T_\ell=T_\ell(1,0)$.

Let
$${\cal T}_{\ell}(r)=\left\{ I_{\ell}(a,b)\,:\,a,b\in
  \{0,1,\cdots, r-1\}, a\ne b \right\}\cup$$
$$ \left\{ T_{\ell}(a,b)\,:\,a,b\in
  \{0,1,\cdots, r-1\}, a\ne b \right\}.
$$

Note that $I_{\ell}=I_{\ell}(1,0)$, $I_{\ell}^c=I_{\ell}(0,1)$,
$T_{\ell}=T_{\ell}(0,1)$,  and $T_{\ell}^c=T_{\ell}(1,0)$.
In \trf{BB}, we do not have $T_{\ell}^c=T_{\ell}(1,0)$ but we note that 
$T_{\ell-1}(0,1)\prec T_{\ell}(1,0)\prec T_{\ell+1}(0,1)$. 
From the point of view of forbidden configurations in the context of
\trf{BB}, $T_{\ell}(1,0)$ and $T_{\ell}(0,1)$ are much the same.  
In general, we have $T_{\ell-1}(a,b)\prec T_{\ell}(b,a)\prec
T_{\ell+1}(a,b)$.

After removing all $T_{\ell}(a,b)$ with $a>b$, we define a reduced
set of configurations:
$${\cal T}'_{\ell}(r)=\left\{ I_{\ell}(a,b)\,:\,a,b\in
  \{0,1,\cdots, r-1\}, a\ne b \right\}
\cup$$
$$ \left\{ T_{\ell}(a,b)\,:\,a,b\in
  \{0,1,\cdots, r-1\}, a< b \right\}.
$$

We have $|{\cal T}'_{\ell}(r)|=3\binom{r}{2}$.
In particular, ${\cal T}'_{\ell}(2)=\{I_{\ell},I_{\ell}^c,T_{\ell}\}$
and
$$\forb(m,r,{\cal T}_{\ell}(r))\leq 
\forb(m,r,{\cal T}'_{\ell}(r))\leq 
\forb(m,r,{\cal T}_{\ell+1}(r)).$$

As a forbidden family, ${\cal T}'_{\ell}(r)$ behaves very much like
${\cal T}_{\ell}(r)$. We will mainly work on ${\cal T}_{\ell}(r)$.
\begin{thm} \label{main} \label{generalizeBB} 
Let $r\geq 2$ be given. Then there exists a constant $c_r$
 so that for any $\ell\geq 1$, 
$$\forb(m,r,{\cal T}_{\ell}(r)) \le
2^{c_r\ell^2}.$$ 
In particular, for $r=2$, we have
$$\forb(m,\{I_{\ell}, I^c_\ell, T_\ell\}) \le
 2^{c_2(\ell^2+\ell)}$$ with the constant $c_2\leq6 \log_2 6<15.51$.
\end{thm}

The previous best known upper bound for $\forb(m,\{I_{\ell},
I^c_\ell, T_\ell\})$ (the case $r=2$) was doubly exponential in $\ell$, so this is a substantial improvement.
The current best known
 lower bound for 
$\forb(m,\{I_{\ell}, I^c_\ell, T_\ell\})$ is still ${\ell}^{c_1\ell}$, and it was
conjectured $\forb(m,\{I_{\ell}, I^c_\ell, T_\ell\})$ $<{\ell}^{c_2\ell}$ 
in \cite{BaBo}. We are  not quite there yet.
The result gives a concrete value for $c_2$ but it is not likely to be best possible.  Currently, for general $r\ge 2$, 
$c_r<30\binom{r}{2}^2\log_2r$. \trf{BB} is also generalized to $r$-colours.

 \trf{BB} yields a  corollary, as noted in \cite{AKoch}, identifying which families yield a constant bound and remarking that all other families yield a linear bound.  Our multicoloured extension of \trf{BB} also yields a similarly corollary.

\begin{cor} Let ${\cal F}=\{F_1,F_2,\ldots ,F_p\}$ and $r$ be given. There are two possibilities. Either $\forb(m,r,{\cal F})$ is $\Omega(m)$ or there exists an $\ell$ and a function 
$f:\, (i,j)\longrightarrow [p]$ defined on all pairs $i,j$ with
$i,j\in\{0,1,\ldots, r-1\}$ and $i\ne j$ so that either
$F_{f(i,j)}\prec I_{\ell}(i,j)$ or $F_{f(i,j)}\prec T_{\ell}(i,j)$,
 in which case  there is a constant $c_{\ell,r}$ with $\forb(m,r,{\cal F})= c_{\ell,r}$.\label{classify}\end{cor}

\proof  Let $F_h$ be $a_h\times {b}_h$ and let $d=\max_{h\in
  [p]}(a_h+{b}_h)$.  Then $F_h\nprec I_{d}(i,j)$  (respectively $F_h\nprec T_{d}(i,j)$)
implies 
$F_h\not\prec I_{m}(i,j)$ (respectively $F_h\nprec T_{m}(i,j)$)
for any $m\ge {d}$. Thus if  for some choice $\,i,j\in \{0,1,\cdots,
r-1\}$ with $i\ne j$ we have $F_h\nprec I_{\ell}(i,j)$  or
$F_h\nprec T_{\ell}(i,j)$ for all $h\in [p]$, then $\forb(m,{\cal
  F})$ is $\Omega(m)$ using the construction  $I_{m}(i,j)$ or $T_{m}(i,j)$. 
\qed
\vskip 10pt
Some further applications of our results and proof ideas are in \srf{applications}.

Our forbidden configurations $I_{\ell}(a,b)$ and $T_{\ell}(a,b)$
 for $a\ne b$  are simple but it is natural to consider non-simple
matrices as forbidden configurations. One natural way to create
non-simple matrices is as follows. For $t>1$, let $t\cdot
M=[M|M|\cdots |M]$, the concatenation of $t$ copies of $M$. For a
family $\cal F$ of matrices, we define $t\cdot {\cal F}=\{t\cdot
M\colon M\in {\cal F}\}$.  In \cite{ALu}, we showed that
$\forb(m,\{t\cdot I_k,t\cdot I_k^c,t\cdot T_k\})$ is $O(m)$.  We
obtain a sharp bound and extend to $r$-matrices.

\begin{thm} \label{generalizetBB}
 Let $\ell\geq 2$, $r\geq 2$, and $t\geq 1$ be given. 
Then there is a constant $c$ with
$$\forb(m,r,t\cdot {\cal T}_{\ell}(r))\leq 2 r(r-1)(t-1)m
  +c.$$
\end{thm}

The proof of the upper bound is in \srf{linear}. For a lower bound consider a choice $a,b\in\{0,1,\ldots ,r-1\}$ with $a\ne b$.  Consider $T_{\ell}(a,b)$.  The first column has (at least) one $b$ and at least one $a$ and the rest either $a$ or $b$. The following easy result is useful.

\begin{thm}\cite{AFS} $\forb(m,t\cdot [0])=\lfloor \frac{tm}{2}\rfloor +1$. \label{row}\end{thm} 

Given a pair $a,b$ we construct $M_m(a,b)$ as follows. Choose $e$ from $\{a,b\}$. Form an $m\times (\lfloor \frac{tm}{2}\rfloor +1)$ matrix $M_m(a,b)$ all of whose entries are $a$ or $b$ with $t\cdot [e]\nprec  M_m(a,b)$ using \trf{row}. 
Now consider the concatenation $M$ of the $\binom{r}{2}$ matrices $M_{m}(a,b)$ over all choices for $a,b$.  Thus $M$ is an $m\times (\binom{r}{2}(\lfloor \frac{tm}{2}\rfloor +1))$ simple $r$-matrix. Moreover  $t\cdot T_{\ell}^{a,b,d}\nprec M$ since the first column of $T_{\ell}(a,b)$ has both $a$'s and $b$'s and so must appear in $M_m(a,b)$. But then 
since  $t\cdot [e]\nprec  M_m(a,b)$ for some choice $e\in\{a,b\}$,  we deduce $t\cdot T_{\ell}(a,b)\nprec  M$. Thus 
$$\forb(m, r,t \cdot {\cal T}_{\ell}(r))\ge \binom{r}{2}(\lfloor \frac{tm}{2}\rfloor).$$
This lower bound is about a quarter of the upper bound in
\trf{generalizetBB}. For $\ell\gg \log_2 t$, a different construction
in \srf{linear} improves this lower bound by a factor of $2$.

\section{Inductive Decomposition}\label{decomp}

Let $M$ be an $m$-rowed matrix. Some notation about repeated columns is needed.
For an $m\times 1$ column $\alpha\in\{0,1,\ldots ,r-1\}^m$, we define $\mu(\alpha,M)$ as the multiplicity of column $\alpha$ in a matrix $M$. At certain points it is important to consider matrices of bounded column multiplicity.   Define a matrix $A$ to be  $s$-\emph{simple} if  every column $\alpha$ of $A$ has $\mu(\alpha,A)\le s$. Simple matrices  are 1-\emph{simple}.  

We need induction ideas from \cite{ALu}. Define
$$\Av(m,r,{\cal F},s)=\left\{A\,:\,A\hbox{ is }m\hbox{-rowed and }s\hbox{-simple }r\hbox{-matrix}, F\nprec A\hbox{ for }F\in{\cal F}\right\},$$
with the analogous definition for $\forb(m,r,{\cal F},s)$. 
The induction proceeds with a matrix in $\Av(m,r,{\cal F},s)$ but the following observation from \cite{ALu} generalized to $r$-matrices shows that the asymptotics of $\forb(m,r,{\cal F})$ are the same as that of $\forb(m,r,{\cal F},s)$:

\begin{E}\forb(m,r,{\cal F})\le \forb(m,r,{\cal F},s)\le s\cdot\forb(m,r,{\cal F}).\label{asymptotics}\end{E}
 The second inequality follows from taking a matrix $A\in\Av(m,r,{\cal F},s)$  and forming the matrix $A'$ where $\mu(\alpha,A')= 1$ if and only if $\mu(\alpha,A)\ge 1$ so that
$\ncols{A'}\le\ncols{A}\le s\cdot \ncols{A'}$.

Let $A\in\Av(m,r,{\cal F},s)$. During the proof of \trf{main}, $s=\left(\frac{r}{2}\right)^i$ for some $i$.  Assume $\ncols{A}=\forb(m,r,{\cal F},s)$. Given a row $r$ we permute rows and columns of $A$ to obtain
\begin{E}A=\begin{matrix}\hbox{ row }r\rightarrow\\ \\ \end{matrix}
\left[\begin{matrix}0\,0\cdots\, 0&1\,1\cdots \,1&\cdots&r-1\,r-1\,\cdots r-1\\ G_0&G_1&&G_{r-1}\\\end{matrix}\right].\label{rdecomp}\end{E}
Each $G_i$ is $s$-simple. Note that typically $[G_0G_1\cdots G_{r-1}]$ is not $s$-simple so we cannot use induction directly on $[G_0G_1\cdots G_{r-1}]$.
We would like to permute the columns of $[G_0G_1\cdots G_{r-1}]$ into
the form of 
$[C_1C_1A_1]$, where $A_1$ is an $s$-simple matrix and $C_1$ is a
matrix such that for each column $\alpha$ of $C_1$ the copies
of $\alpha$ in the two copies of $C_1$ comes from different $G_i$ and $\mu(\alpha, [C_1C_1A_1])> s$.
This can be done greedily. Initially set $A_1=[G_0G_1\cdots
G_{r-1}]$. If $A_1$ is not $s$-simple, then there is a
column $\alpha$ appear in various $G_i$. Move $\alpha$ to $C_1$ and
delete two copies of $\alpha$ (in two different $G_i$) from $A_1$. When the process stops,
we get the matrix $A_1$ and $C_1$ as stated.
Note $A_1$ is $s$-simple. For each column $\alpha$ of $C_1$, the
multiplicity of $\alpha$ satisfies
$$s< \mu(\alpha, [C_1C_1A_1])\leq rs.$$

We inductively apply the same decomposition to the $(m-1)$-rowed
$s$-simple matrix $A_1$ and get a $(m-2)$-rowed $s$-simple matrix
$A_2$ and an $(m-2)$-rowed matrix $C_2$ appearing twice, and so on.

We deduce that
\begin{E}\ncols{A}\le \left(2\sum_{i=1}^{m-1}\ncols{C_i}
\right)  + rs\label{induction}\end{E}
where $rs$ is the maximum  number of columns in any 1-rowed $s$-simple matrix.
Note that  each $C_i$ is $\frac{rs}{2}$-simple. Note that it is possible to have $\ncols{C_i}=0$ in which case we ignore such cases.
The idea is to have $A$ as a root of a tree with the children $C_i$ for each $i$ with $\ncols{C_i}>0$.

\section{Proof of Theorem \ref{main}}

We denote by $R(k_1,k_2,\ldots, k_n)$ the multicolour Ramsey number for the minimum number $p$  of vertices of $K_p$ so that when the edges are coloured  using $n$ colours $\{1,2,\ldots ,n\}$ there will be a monochromatic clique of colour $i$ of size $k_i$ for some $i\in\{1,2,\ldots ,n\}$. 
It is well-known that the multicolour Ramsey number satisfyies
the  following inequality:
$$R(k_1,k_2,\ldots, k_n)\leq \left(\begin{matrix}\sum_{i=1}^n(k_i-1)\\ k_1-1,\ldots,k_n-1\\ \end{matrix}
\right)   <n^{\sum_{i=1}^n(k_i-1)}.$$
This follows from showing that the Ramsey numbers satisfy the same recurrence as the multinomial coefficients but with smaller base cases.   The proof of \trf{main}  uses the following $2r^2-r$ multicolour Ramsey number.
  
  \begin{prop}
\begin{E}\hbox{ Let }\quad u=R(\underbrace{(r-1)(\ell-1)+1,\ldots, (r-1)(\ell-1)+1}_{r \mbox{
    copies}}, \underbrace{2\ell,\ldots, 2\ell}_{2r(r-1) \mbox{ copies}}).\label{udef}\end{E}
Then a  upper bound on $u$ is:
\begin{E}u\leq (2r^2-r)^{r(r-1)(5l-3)}<r^{15r(r-1)l}.\label{ubound}\end{E}\end{prop}
We highlight the fact that $u$ is bounded by a single exponential function in
$\ell$ for a fixed $r$ (independent of $s$ and $m$). 
\vskip 10pt

 We are now going to describe a tree growing procedure used in the proof. We initially start with some $A\in\Av(m,r,{\cal T}_\ell(r))$ as the root of the tree.  Given a matrix $A\in \Av(m',r,{\cal T}_\ell(r),s)$ that is a node in our tree, we apply the induction ideas of \srf{decomp} to obtain matrices $C_1, C_2, \ldots C_{m'-1}$ and we set the children of $A$ to be the matrices $C_i$ for those $i$ with $\ncols{C_i}>0$. Note that $C_i\in\Av(m'',r,{\cal T}_\ell(r),\frac{rs}{2})$.    Repeat. 
 
 \begin{lemma}Given $A\in\Av(m,r,{\cal T}_\ell(r))$,  form a tree as described above.  Then the depth of the tree is at most $\binom{r}{2}(\ell-1)+1$. \label{depth}\end{lemma}
 
 \proof Suppose  there is a chain
of depth $\binom{r}{2}(\ell-1)+2$ in the tree. Pick any column vector $\alpha$ in the matrix forming the terminal node. At its parent node (or row), $\alpha$ is extended twice with some
choices $a_i, b_i$ ($a_i\ne b_i$). We label this edge (of the chosen
chain) by the colour $\{a_i,b_i\}$. Since the number of  colours (each consisting of a pair from $\{0,1,\ldots ,r-1\}$) is at most
$\binom{r}{2}$, there is some pair $\{a,b\}$
occurring at least $\ell$ times (by Pigeonhole principle). As the result, $\alpha$ can be extended
into $2^{\ell}$ columns so that the columns form a submatrix $B$ of $A$  that contains the complete 
 $\ell\times 2^{\ell}$-configuration using only two colours $a$ and
 $b$. In particular, $A$ contains $I_\ell(a,b)$ (as well as $T_\ell(a,b)$). This contradiction completes the proof.\qed

 \begin{lemma}Given $A\in\Av(m,r,{\cal T}_\ell(r))$, form a tree as described above. Then the maximum branching is at most $u$ with $u$ given in \rf{udef}.\label{ulemma} \end{lemma} 
\proof $A\in\Av(m,r,{\cal T}_\ell(r),s)$ be a node of the tree and determine the children of $A$ as described above.  Let $\left|\{i\colon C_i\ne\emptyset\}\right| \ge u$. 
Selecting one column $\c_i$ from each non-empty $C_i$ and deleting some rows
if necessary, we get the following submatrix of $A$.

\begin{E}\begin{array}{|cc|cc|cc|c}
a_1&b_1&*&*&*&*&\cdots\\
\cline{1-4}
&&a_2&b_2&*&*&\\
\cline{3-6}
&&&&a_3&b_3&\\
\cline{5-6}
&&&&&&\\
\c_1 &\c_1&\c_2&\c_2&\c_3&\c_3&\cdots\\
\end{array}\label{eq:cs}\end{E}
On the diagonal, we can assume $a_i<b_i$ for each $i$.
This is a $u\times 2u$ matrix.
We can view this as a $u\times u$ ``square'' matrix with
each entry is a $1\times 2$ row vector. 
Note that in this ``square'' matrix,
the entries below the diagonal are special $1\times 2$ row vectors
of the form $(x,x)$ while the $i$-th diagonal entries is $(a_i,b_i)$
satisfying $a_i<b_i$. There is no restriction on the 
entries above the diagonal.

Now we form a colouring of the complete graph $K_{u}$.
For each edge $ij\in E(K_{u})$ (with $i<j$), set the colour of
$ij$ to be the combination of the $(i,j)$ entry and the $(j,i)$ entry.
Write the $(i,j)$ entry on the top of $(j,i)$ entry to form a
$2\times 2$ matrix, which has the following generic form:
$\big( \begin{array}{cc}
  y_1 & y_2\\
x & x
\end{array}
\big)$.

There are at most $r^3$ such $2\times 2$ matrices and so $r^3$ colours on which to apply a 
multicolour version of the Ramsey's theorem.  We can reduce the number of colours to $2r^2-r$, and obtain a better upper bound, by
combining some patterns of $2\times 2$ matrices into one colour class
to reduce the total number of colours needed. To be precise, we define
the colour classes as
$$
\bigcup_a
\left\{\left( \begin{array}{cc}
 a & a\\
a & a
\end{array}
\right) \right\}
\cup 
\bigcup_{a\ne b}
\left\{
\left( \begin{array}{cc}
 b & *\\
a & a
\end{array}
\right)
\right\}
\cup 
\bigcup_{a\ne b}
\left\{
\left( \begin{array}{cc}
 * & b\\
a & a
\end{array}
\right)\right\}
$$
Note that the matrix $\left( \begin{array}{cc}
  b_1 & b_2\\
a & a
\end{array}
\right)$ for $b_1\ne b_2$  fits two colour classes. When this occurs, we break the tie
arbitrarily. 

A critical idea here is that we only apply Ramsey's theorem once
to get a uniform pattern for both entries below and above the diagonal!
By the definition of $u$ as a Ramsey number \rf{udef},  one of the following cases must happen.
\vskip 10pt
\noindent {\bf Case 1:}  There is a number $a\in \{0,1,\ldots, r-1\}$ such that
there is a monochromatic clique of size $(r-1)(\ell-1)+1$ using colour $\left( \begin{array}{cc}
 a & a\\
a & a
\end{array}
\right)
$.

Since the diagonals have two colours, we can pick one colour other than
$a$ and get a square matrix so that all off-diagonal entries are $a$'s
and all diagonal elements are not equal to $a$. Since this matrix has
$(r-1)(\ell-1)+1$ rows, by pigeonhole principle, there is a colour,
call it  $b$, appearing at least $\ell$ times on the diagonal. This gives a submatrix
$I_\ell(a,b)$ in $A$, contradicting $A\in\Av(m,r,{\cal T}_\ell(r),s)$. This eliminates Case 1.
\vskip 10pt

\noindent {\bf Case 2:}  There is a pair $a\ne b  \in \{0,1,\ldots, r-1\}$ such that
there is a monochromatic clique of size $2\ell$ using colour 
$\left( \begin{array}{cc}
 b & *\\
a & a
\end{array}
\right)$.

By selecting first column from each $1\times 2$ entry, we obtain
a $(2\ell \times 2\ell)$-square matrix so that the entries below the
diagonal are all $a$'s and the entry above the diagonal are all $b$'s.
The diagonal entries are arbitrary. By deleting first column,
second row, third column, fourth row, and so on, we get a submatrix
$T_\ell(a,b)$ in $A$, contradicting $A\in\Av(m,r,{\cal T}_\ell(r),s)$. This eliminates Case 2.
\vskip 10pt

\noindent {\bf Case 3:}  There is a pair $a\ne b  \in \{0,1,\ldots, r-1\}$ such that
there is a monochromatic clique of size $2\ell$ using colour 
$\left( \begin{array}{cc}
 * & b\\
a & a
\end{array}
\right)$.

This is symmetric to Case 2, and thus it can be eliminated in the same way. 

Thus such an $A$ with $u$ children does not exist. \qed

\vskip 10pt
{\bf Proof of \trf{generalizeBB}}:
We do our tree growing beginning with some \hfil\break$A\in\Av(m,r,{\cal T}_\ell(r))$. 
 We will be applying \rf{induction}. Regardless the value of $m$, at most $u$ summands in the
summation above are non-zero by \lrf{ulemma}. It is sufficient to bound each $\ncols{C_i}$.
Recall that each $C_i$ is $\frac{rs}{2}$-simple when the parent node is $s$-simple.

For $i=0,1,2,
\ldots, \binom{r}{2}(\ell-1)+1$,
let $f(i)$ be the maximum value of $\ncols{C}$ in the $i$-th depth node of matrix $C$  in the
tree above.  By convention $f(0)=\ncols{A}$.
 Inequality \eqref{induction} combined with \lrf{ulemma} implies
 the following recursive formula:
$$f(i)\leq 2u f(i+1)+r\left(\frac{r}{2}\right)^i.$$
By \lrf{depth}, we have the initial condition $f(\binom{r}{2}(\ell-1)+1)\leq 
r\cdot (\frac{r}{2})^{\binom{r}{2}(\ell-1)+1}$ by \rf{induction}, where a matrix in a node of the tree at depth $\binom{r}{2}(\ell-1)+1$ is $(\frac{r}{2})^{\binom{r}{2}(\ell-1)+1}$-simple.

Pick a common upper bound, say $r^{15r(r-1)\ell}$, for both
$2u$ and $r(\frac{r}{2})^{\binom{r}{2}(\ell-1)+1}+1$.
It implies 
$$f(i)+1\leq (f(i+1)+1)r^{15r(r-1)\ell}.$$
Thus
$$\ncols{A}< f(0)+1\leq \left(r^{15r(r-1)\ell}\right)^{
\binom{r}{2}(\ell-1)+2}
\leq r^{30\binom{r}{2}^2\ell^2}.$$

For the special case $r=2$, the upper
bound can be reduced. First, each diagonal entry of the matrix in Equation \eqref{eq:cs}
is always $[0\; 1]$. In the proof \lrf{ulemma}, Case 2 and Case 3, there is
no need to delete rows and columns alternatively. The size
of the monochromatic clique can be taken to $\ell$ instead of $2\ell$.
Thus in this setting we may take $u=R(\ell,\ell,\ell,\ell,\ell,\ell)=R_6(\ell)$ and obtain
$$|\{i\colon C_i\ne \emptyset\}|<R_6(\ell).$$
Second, all $C_i$ are simple matrices ($\frac{r}{2}=1=s$). So the recursive formula for
$f(i)$ is
$$f(i)\leq 2 R_6(\ell) f(i+1)+2$$
with the initial condition $f(\ell)\leq 2$.
It is not hard to check that  $f(\ell-1)\leq 2\ell
-1$ for $\ell\geq 2$. (Otherwise, we get a row vector consisting of
$\ell$ $0$'s or $\ell$ $1$'s, which can be used to extend into 
$T_\ell(0,1)$ or $T_\ell(1,0)$.)
 Use the bound
$R_6(\ell)<6^{6(\ell-1)}$
and solve the recursive relation for $f(i)$. We get
$$f(0)<\left(2\cdot 6^{6(\ell-1)}\right )^{\ell-1}\cdot(2\ell-1)
\leq  6^{6(\ell-1)\ell}.$$
Thus,
$$\forb(m,2,{\cal T}_\ell(2)) <6^{6(\ell-1)\ell}.$$
This implies
$$\forb(m, \{I_\ell,I^c_\ell, T_\ell\})
\leq \forb(m,2,{\cal T}_{\ell+1}(2))<6^{6\ell(\ell+1)},$$
yielding the stated bound.
\qed

\section{A non-simple forbidden family}\label{linear}

In many examples when computing $\forb(m,t\cdot F)$, the proof ideas for $\forb(m,F)$ are important.
A much weaker linear bound than \trf{generalizetBB} for $r=2$ is in
\cite{ALu}  (the constant multiplying $m$ is the constant of
\trf{BB}).   The upper bound of \trf{generalizetBB} is only off by a
factor of $2$ from the lower bound asymptotically.
 In fact,  \trf{generalizetBB} can be generalized to
 $s$-simple matrices. 

\begin{thm} \label{generalizetsBB}
Let $\ell\geq 2$, $r\geq 2$,  $t\geq 1$, and $s\geq t-1$ be given. 
Then there is a constant $c_{\ell,r,t}$ with
$$\forb(m,r,t\cdot {\cal T}_{\ell}(r),s)\leq 2r(r-1)(t-1)m
  +c_{\ell,r,t} +rs.$$
\end{thm}

This upper bound is only off by a
factor of $2$ from the lower bound. Note that we can use this bound with $s<t-1$ as noted in \rf{asymptotics}.
This yields \trf{generalizetBB}. The proof appears below.
\begin{thm}
  Let $\ell\geq 3$, $r\geq 2$,  $t\geq 1$, and $s\geq t-1$ be given. 
Then 
$$\forb(m,r,t\cdot {\cal T}_{\ell}(r),s)\geq r(r-1)(t-1)m.$$
\end{thm}
\noindent{\bf Proof}:
Consider the matrix $M_m$ obtained by the concatenation of all
matrices in $(t-1)\cdot \{ I_m(a,b)\colon \hbox{ for }a,b\in\{0,1,\ldots,r-1\}, a\ne b\}$. 
The matrix $M_m$ is $(t-1)$-simple and hence  $s$-simple with
$$\ncols{M_m}=r(r-1)(t-1)m.$$
It suffices to show that $M_m\in \Av(m,r, t\cdot {\cal T}_\ell(r))$.
For a choice $a,b\in\{0,1,\ldots,r-1\}$ with $a\ne b$, we need show 
$t\cdot I_{\ell}(a,b) \not \prec M_m$ and $t\cdot
T_{\ell}(a,b)\not \prec M_m$.

Suppose not, say $t\cdot I_{\ell}(a,b) \prec M_m$. There are a list
of $\ell$-rows $i_1,i_2,\ldots, i_{\ell}$ and $t\ell$ columns
$j_1^1,\cdots, j_r^1,j_1^2,\cdots, j_r^2,\cdots, j_1^t,\cdots, j_r^t,$
evidencing the copies of $t\cdot  I_{\ell}(a,b)$ in $M_m$.
Let us restrict to the rows $i_1,i_2,\ldots, i_{\ell}$ at the moment.
For each row $i_h$, there are $t$ columns who has 
 $a$ at row $i_h$ and $b$ at other $\ell-1$ rows.
These columns can only show up in exactly one column in each
copy of $I_m(a,b)$. 
But we only have $t-1$ copies of
$T_m(a,b)$. Thus  $t\cdot I_{\ell}(a,b) \not\prec M_m$. 

Similarly suppose  $t\cdot T_{\ell}(a,b) \prec M_m$. With $\ell\geq 3$,
then there are $t$ columns of $M_m$ and $\ell$ rows $i_1,i_2,\ldots ,i_{\ell}$  containing $a$ in  row $i_1$ and containing
$b$'s in the other $\ell-1$ rows.  Each copy of $I_m(a,b)$ contains exactly
one such columns with an $a$ in row $i_1$ and all other copies of $I_m(c,d)$ do not contain
such columns of $a$'s and $b$'s. But there are only $t-1$ copies of
$I_m(a,b)$. Thus  $t\cdot T_{\ell}(a,b) \not\prec M_m$. 
\qed

\vskip 10pt
Note that in this construction of the lower bound, every column has
 equal multiplicity $t-1$. By adding $q:=\lceil log_2(t-1)\rceil$ rows  we can distinguish $t-1$ columns and  obtain a $(m+q)$-rowed simple matrix $M'\in \Av(m+q,r, t\cdot {\cal T}_{\ell-q})$. Thus we have the following corollary.

\begin{cor} 
Let $t\geq 1$, $r\geq 2$, and $\ell\geq \lceil\log_2t\rceil+3$ be given. There is a constant $c$ with  
$\forb(m,r,t\cdot {\cal T}_{\ell}(r))\geq r(r-1)(t-1)m -c.$ \qed
\end{cor}

\vskip 10pt
We begin the proof of \trf{generalizetsBB}. Consider $(t-1)$-simple matrices. 
Consider some $A\in\Av(m,r, t\cdot {\cal T}_{\ell}(r))$ so that 
$A\in\Av(m,r, t\cdot {\cal T}_{\ell}(r),t-1)$ . Apply our  inductive decomposition of \srf{decomp}  with the bound \rf{induction}.

 Keep track of the matrices $C_i$  that are generated yielding the following structure in $A$:

\begin{E}\begin{array}{|c|c|c|c|c|c|c}
a^1_1\,a^1_{2}\,\cdots\,&b^1_{1}\,b^1_2\,\cdots\,&*&*&*&*&\cdots\\
\cline{1-4}
&&a^2_1\,a^2_2\,\cdots\,&b^2_1\,b^2_2\,\cdots\,&*&*&\\
\cline{3-6}
&&&&a^3_1\,a^3_2\,\cdots\,&b^3_1\,b^3_2\,\cdots\,&\\
\cline{5-6}
&&&&&&\\
C_1 &C_1&C_2&C_2&C_3&C_3&\cdots\\
\end{array}\label{Cs}\end{E}

    In what follows let 
\begin{E}T=r(r-1)(t-1)+1.\label{Tdef}\end{E}
By the construction, we may require $a^i_j< b^i_j$ for all $i,j=1,2,\ldots, m-1$.
In analogy to $u$ of \rf{udef},  let $v$ be the multicolour Ramsey number:
\begin{E}v=R(\,\underbrace{(r-1)(\ell-1)+1,\ldots, (r-1)(\ell-1)+1}_{r^T \mbox{
    copies}},\underbrace{2t\ell,\ldots, 2t\ell}_{2Tr(r-1) \mbox{ copies}}).\label{vdef}\end{E}

\begin{lemma} In the inductive structure of \rf{Cs}, we have
$\left|\{i\colon \ncols{C_i}\ge T\}\right| < v
$.\label{vbd}\end{lemma}

\proof
Assume  
\begin{E}|\{i\,:\,\ncols{C_i}\ge T\}|\ge v.\label{vbound}\end{E} 
In what follows we arrive at a contradiction. 
 $A$ has a structure as in \rf{Cs}. Select  rows $i$ for which $\ncols{C_i(a,b)}\ge T$ to obtain an $v$-rowed matrix as follows. For a given $i$,  select $T$ columns from $C_i(a,b)$ and  we have
$$
\left[\begin{array}{cc}a^i_{1} a^i_{2}\cdots a^i_{T} &b^i_{1}b^i_{2}\cdots b^i_{T}\\
\alpha&\alpha\\
\beta&\beta\\
\delta&\delta\\
\vdots&\vdots\\
\end{array}\right]
\hbox{ is a submatrix of }
\left[\begin{array}{cc}a^i_{1}a^i_{2}\cdots &b^i_{1}b^i_{2}\cdots \\
C_i&C_i\\
\end{array}\right],
$$
where each entry $\alpha,\beta,\ldots$ and $[a^i_1a^i_2\cdots a^i_T]$ and $[b^i_1b^i_2\cdots b^i_T]$ are $1\times T$ row vectors.
Now with $v$ choices $i$ with $\ncols{C_i}\ge T$ we obtain  a $v\times 2Tv$ $r$-matrix $X$ as follows:
$$\left[\begin{array}{@{}cc@{}|@{}cc@{}|@{}cc@{}|@{}cc@{}|c@{}}
a^{i_1}_1\cdots a^{i_1}_T&b^{i_1}_1\cdots b^{i_1}_T&*&*&*&*&*&*&\cdots\\
\alpha&\alpha&a^{i_2}_1\cdots a^{i_2}_T&b^{i_2}_1\cdots b^{i_2}_T&*&*&*&*&\cdots\\
\beta&\beta&\gamma&\gamma&a^{i_3}_1\cdots a^{i_3}_T&b^{i_3}_1\cdots b^{i_3}_T&*&*&\cdots \\
\delta&\delta&\mu&\mu&\nu&\nu&a^{i_4}_1\cdots a^{i_4}_T&b^{i_4}_1\cdots b^{i_4}_T&\\
\vdots&\vdots&\vdots&\vdots&\vdots&\vdots&&\\
\end{array}\right],$$
where the entries $\alpha,\beta,\ldots$ are $1\times T$ row vectors with entries from $\{0,1,\ldots ,r-1\}$. 

View this matrix as a $v\times v$ ``square'' matrix with each entry being a $1\times 2T$ row vector.
 All entries below the diagonal are ``doubled''
row vectors, i.e., the concatenation of two identical 
$1\times
T$ row vectors.
All diagonal entries are 
the concatenation of two $1\times T$ row vectors,
where each coordinate of the first vector is always strictly less than
the corresponding coordinate of the second vector.
There is no restriction on the entries above the diagonal.

Now we form a colouring of the complete graph $K_{v}$.
For each edge $ij\in E(K_{v})$ (with $i<j$), colour
$ij$ using the combination of the $(i,j)$ entry and the $(j,i)$ entry.
Write the $(i,j)$ entry on the top of $(j,i)$ entry to form a
$2\times 2T$ matrix, which has the following generic form:
$\left( \begin{array}{cc}
  \beta_1 & \beta_2\\
\alpha & \alpha
\end{array}
\right)$.
Here $\alpha,\beta_1,\beta_2$ are $1\times T$ row vectors.

There are at most $r^{3T}$ such matrices. Instead of applying 
Ramsey's theorem with $r^{3T}$ colours, we can reduce the total number of colours needed  by
combining some patterns of $2\times 2T$ matrices into a single colour class.

The first type of colour classes is denoted $C(\alpha)$
(with $\alpha\in \{0,1,\ldots,r-1\}^T$) and
consists of patterns
$\left( \begin{array}{cc}
 \alpha & \alpha\\
\alpha & \alpha
\end{array}
\right)$.
The second type of colour classes is denoted $C(a,b,i)$
(with $a\ne b\in  \{0,1,\ldots,r-1\}^T$ and $1\leq i\leq 2T$)
consists of patterns
$\left( \begin{array}{cc}
 \beta_1 & \beta_2\\
\alpha & \alpha
\end{array}
\right)$  whose $i$-th column is the vector 
$\left(
  \begin{array}[c]{c}
     b\\
     a\\
  \end{array}
\right)$.

A $2\times 2T$ matrix may fit multiple colour classes.
 When this occurs, we break the tie
arbitrarily. The total number of colours are 
$r^T+r^2T$ (reduced from $r^{3T}$). By the definition of $v$ and \rf{vbound}, one of the following cases must happen.
\vskip 10pt

\noindent {\bf Case 1:} There is a number $\alpha \in \{0,1,\ldots, r-1\}^T$ such that
there is a monochromatic clique of size $(r-1)(\ell-1)+1$ using colour 
$\left( \begin{array}{cc}
 \alpha & \alpha\\
\alpha & \alpha
\end{array}
\right)
$.

We get the following $((r-1)(\ell-1)+1)$-rowed submatrix:
$$\left[\begin{matrix}
**&\alpha\alpha&\alpha\alpha&\alpha\alpha&\cdots\\
\alpha\alpha&**&\alpha\alpha&\alpha\alpha&\cdots\\
\alpha\alpha&\alpha\alpha&**&\alpha\alpha&\cdots\\
\alpha\alpha&\alpha\alpha&\alpha\alpha&**\\
\vdots & \vdots &\vdots\\
\end{matrix}\right]$$

Since all the diagonal entry have two choices, we can pick one colour other than
the corresponding colour in $\alpha$.
We get the following sub-matrix:
$$\left[\begin{matrix}
\beta_1&\alpha&\alpha&\alpha&\cdots\\
\alpha&\beta_2&\alpha&\alpha&\cdots\\
\alpha&\alpha&\beta_3&\alpha&\cdots\\
\alpha&\alpha&\alpha&\beta_4&\cdots\\
\vdots & \vdots &\vdots&\vdots\\
\end{matrix}\right]$$
where the diagonal entry $\beta_i(j)\ne \alpha(j)$
for any $i$ and any $j=1,2,\ldots, T$.

Using \rf{Tdef}, the Pigeonhole principle yields a colour $a$ appearing
in  $\alpha$ at least $(r-1)(t-1)+1$ times.
By selecting these columns we get an  $((r-1)(\ell-1)+1)$-rowed submatrix
$$\left[\begin{matrix}
* \cdots *&a\cdots  a&a\cdots a&a\cdots a&\cdots\\
a\ldots a&*\cdots *&a\cdots  a&a\cdots  a&\cdots\\
a\cdots  a&a\cdots a&*\cdots  *&a\cdots  a&\cdots\\
a\cdots  a&a\cdots  a&a\cdots a &*\cdots *&\cdots\\
\vdots & \vdots &\vdots&\vdots\\
\end{matrix}\right].$$
Note that all diagonal elements (marked by $*$) are not equal to $a$. By Pigeonhole
principle, each diagonal entry has one colour $b_i\ne a$ appearing  at
least $t$ times. Since the number of row s is $(r-1)(\ell-1)+1$,
among those $b_i$'s, there is a colour $b$ appears in at least $\ell$
rows by Pigeonhole
principle. This gives a configuration $t\cdot I_{\ell}(b,a)$, contradicting $A\in\Av(m,r, t\cdot {\cal T}_{\ell}(r))$. 

\vskip 10pt
\noindent {\bf Case 2:}  There is a pair $a\ne b  \in \{0,1,\ldots, r-1\}$
and an index $i\in\{1,2,\ldots, 2T\}$
 such that
there is a monochromatic clique of size $2t\ell$ using colour $C(a,b,i)$.

By selecting $i$-th column from each $1\times 2T$ entry, we obtain
a $2t\ell \times 2t\ell$-square submatrix of $A$ so that the entry below the
diagonal are all $a$'s and the entry above the diagonal are all $b$'s.
The diagonal entries could be arbitrary. By deleting first column,
second row, third column, fourth row, and so on, we get a submatrix
$T_{t\ell}(a,b)$ of $A$ and this contains $t\cdot T_\ell(a,b)$, contradicting $A\in\Av(m,r, t\cdot {\cal T}_{\ell}(r))$. 

Both cases end in a contradiction so we may conclude  $\left|\{i\colon \ncols{C_i}\ge T\}\right| < v$. \qed

\vskip 10pt
\noindent{\bf Proof of \trf{generalizetsBB}}:\hskip .5cm 
We consider $(t-1)$-simple matrices. 
Consider some $A\in\Av(m,r, t\cdot {\cal T}_{\ell}(r))$ so that 
$A\in\Av(m,r, t\cdot {\cal T}_{\ell}(r),t-1)$. Obtain the inductive structure of \rf{Cs} with the bound \rf{induction}.

By \lrf{vbd}, $|\{i\,:\,\ncols{C_i}\ge T\}|\le v$ with $T$ given in \rf{Tdef}.
For each $i$, let $C_i'$ denote the simple matrix obtained from $C_i$ by reducing multiplicities to 1. Then 
$\ncols{C_i'}\le \forb(m-i, r, {\cal T}_\ell(r))$
since the multiplicity of each column $\alpha$ (of $C_i$)
 in $C_iC_iA_i$ is at least
$s+1\geq t$. By Theorem \ref{main}, 
 there are at most $2^{c_r\ell^2}$ distinct columns in each
$C_i$. Since $C_{i}$ is $\frac{rs}{2}$-simple, we have
$$\ncols{C_i}\leq \frac{rs}{2} \cdot 2^{c_r\ell^2}.$$
We obtain using \rf{induction}
\begin{align*}
\ncols{A}&\le 2\left(\sum_{i\,:\,\ncols{C_i}< T}\ncols{C_i}+\sum_{i\,:\,\ncols{C_i}\ge T}\ncols{C_i}\right)+rs\\
&\le 2(T-1)m +2u\cdot \frac{rs}{2} \cdot 2^{c_r\ell^2}+rs\\
&=2r(r-1)(t-1)m+c, 
\end{align*}
for a  constant $c$ depending on $r,s,\ell$ (note $s\ge t-1$). This is the desired bound, albeit with $c$ being quite large.\qed

\section{Applications}\label{applications}
We can apply our results in several ways. The following two variations for the problem of forbidden configurations are noted in \cite{survey}.

\noindent {\bf Fixed Row order for Configurations:}
There have been some investigations for cases where only column permutations of
are allowed. Note in our proof, the row order is fixed. So Theorem
\ref{main} works for this variation with the exact same
upper bound.

\noindent {\bf Forbidden submatrices:}
When both row and column orders are fixed, this is the problem of
forbidden submatrices. 
 Let $I_\ell^R(a,b)$ and $T^R_\ell(a,b)$)
 be the matrix obtained from 
$I_\ell(a,b)$ and $T_\ell(a,b)$ respectively
 by reversing the column order.
Our first observation is that
any matrix in the four family $\{I_\ell(a,b)\}_{\ell\geq 3}$,
$\{I^R_\ell(a,b)\}_{\ell\geq 3}$, $\{T_\ell(a,b)\}_{\ell\geq 3}$,
and $\{T^R_\ell(a,b)\}_{\ell\geq 3}$ cannot be the submatrix of the one
in another family. We have the following submatrix version of \trf{main}.

\begin{thm}
  For any $r\geq 2$, there is a constant $c_r$ so that for any
  $\ell\geq 2$
and any matrix with entries drawn from $\{0,1,...,r -1\}$
with at least $2^{c_r\ell^4}$ different columns must contain 
a submatrix $I_\ell(a,b)$, $T_\ell(a,b)$, $I^R_\ell(a,b)$, or
$T^R_\ell(a,b)$ for some $a\ne b\in \{0,1,...,r -1\}$.
\end{thm}

\begin{proof}
Let $c_r=30{r\choose 2}^2\log_2r$. Let $A$ be a simple $r$-matrix with $\ncols{A}> 2^{c_r\ell^4}$ and apply Theorem \ref{main}.
Then $A$ has an ${\ell}^2\times {\ell}^2$ submatrix $F$ which is a column permutation of  $I_{\ell^2}(a,b)$ or $T_{\ell^2}(a,b)$. Let the column permutation be $\sigma$. By the fundamental Erd\H{o}s-Szekeres Theorem,   any sequence of
$(\ell-1)^2+1\le {\ell}^2$
distinct numbers must contain a monotone subsequence of $\ell$
numbers. Let $i_1<i_2<\cdots<i_\ell$ be the indexes so that the subsequence
$\sigma(i_1), \sigma(i_2),\ldots, \sigma(i_\ell)$ is either increasing 
or decreasing. Consider the submatrix of $F$ obtained from $F$ by
restricting it to the $i_1, i_2, \ldots, i_\ell$ rows and 
$\sigma(i_1), \sigma(i_2),\ldots, \sigma(i_\ell)$ columns. Then
we obtained a submatrix which is one of the four matrices:
 $I_\ell(a,b)$, $T_\ell(a,b)$, $I^R_\ell(a,b)$, or
$T^R_\ell(a,b)$.
\qed
\end{proof}
\vskip 10pt






We also can obtain some interesting variants of \trf{main} by replacing some of the matrices in ${\cal T}_{\ell}(r)$. 
As noted in \cite{FS}, we must forbid at least one $(a,b)$-matrix  for
each pair $a,b\in\{0,1,\ldots ,r-1\}$ in order to have a polynomial
bound. What follows provides some additional examples of forbidden
families related to ${\cal T}_{\ell}(r)$ with interesting polynomial
bounds.  We define $\forbmax(m,r,{\cal F})=\max_{m'\le
  m}\forb(m',r,{\cal F})$. It has been conjectured that
$\forbmax(m,{\cal F})=\forb(m,{\cal F})$ for large $m$ and for many
${\cal F}$ this can be proven.  We have the following theorem.

\begin{thm}Let $r,\ell$ be given and let $\pi=P_0\cup P_1\cup
  \cdots\cup P_{t-1}$ be a partition  of $\{0,1,\ldots ,r-1\}$ into
  $t$ parts. There is a constant $c_{\ell, r}$ such that
for  any 
  family of matrices ${\cal F}_i$
 all of whose entries lie in $P_i$ ($1\leq i\leq t$) 
$$\forb(m,r,\left\{{\cal T}_{\ell,\pi}(r)\cup\bigcup_{i=0}^{t-1} {\cal F}_i\right\})
\le c_{\ell,r}\cdot \prod_{i=0}^{t-1}\forbmax(m,|P_i|,{\cal F}_i).$$
$$
\hbox{Here }{\cal T}_{\ell,\pi}(r)
=\left\{ I_{\ell}(a,b)\,:\,a\in P_i, b\in P_j,i\ne j 
  \right\}\cup \left\{ T_{\ell}(a,b)\,:\,a\in P_i, b\in P_j,i\ne j \right\}$$
  $$
  \cup \left\{ T_{\ell}(b,c)\,:\,b,c\in P_j, b\ne c \right\}.$$

\label{allbuts}\end{thm}

\proof Consider $A\in\Av(m,r,{\cal
  T}_{\ell,\pi}(r)\cup\bigcup_{i=0}^{t-1} {\cal F}_i)$. Now form a
$t$-matrix $A_{\pi}$ from $A$ by replacing  each entry $a$ of $A$ that
is in $P_i$ by the entry $i$. Of course $A_{\pi}$ is typically not
simple but the maximum number of different columns is finite.
Let $k= R_{r^2}(2\ell)$ and
$c_{\ell,r}=2^{c_t k^2}$ where $c_t$ is the constant
specified in Theorem \ref{main}. Otherwise  $A_{\pi}$ contains a
configuration $I_{k}(i,j)$ or $T_{k}(i,j)$ for 
some $i\ne j \in \{0,1,\ldots, t\}$. Now we return the colours $i$ and
$j$ to the original colours in $P_i$ and $P_j$ respectively. We obtain
a $k\times k$ matrix $F\prec A$ of one of the following two types.

\begin{description}
  \item [type $I_{k}(i,j)$]: All diagonal entries of $F$ are 
    in $P_i$ while all off-diagonal entries are in $P_j$.
First we apply the Pigeonhole principle on the diagonal and get
a square submatrix $F_1$ of size $k/|P_i|$ so that all diagonal 
entries have the common value, say $a$, in $P_i$. Then we apply
the multicolour Ramsey Theorem to $F_1$, where $F_1$ is viewed
as a $|P_j|^2$-colouring of the complete graph on $k/|P_i|$ vertices with edge $(x,y)$ (for $x<y$) 
 coloured $(b,c)$ if the $(x,y)$ entry of $F_1$ is $c$ and the $(y,x)$ entry of $F_1$ is $b$ .
Since $k/|P_i|>R_{|P_j|^2}(2\ell)$,
there exists a monochromatic clique of size $2\ell$ in $F_1$. Say
the colour is  $(b,c)$, where $b,c\in P_j$.  If $b=c$, we obtain $I_{2\ell}(a,b)$ and so $I_{\ell}(a,b)$, a contradiction. 
If $b\ne c$, we obtained a $2\ell\times 2\ell$ matrix with $a$'s on the diagonal, $b$'s below the diagonal and $c$'s above the diagonal. Forming the submatrix consisting of the odd indexed rows and the even indexed columns, we obtain 
$T_{\ell}(b,c)$, a contradiction.

 \item [type $T_k(i,j)$:]  All entries below diagonal of $F$ 
are in $P_i$ while the rest of entries are in $P_j$. We apply the
multicolour Ramsey Theorem to $F$ to obtain a submatrix $F_2$ of size $2\ell$
whose lower-diagonal entries has a common value $a\in P_i$ and  whose
upper-diagonal entries has a common value $b\in P_j$. There is no
restriction on the diagonal of $F_2$. We can get $T_\ell(a,b)$ from
$F_2$ by deleting the first column, the second row, and so on. Again we have a contradiction.
\end{description}

Thus the number of different columns in $A_{\pi}$ is bounded by a constant. 
Now consider $\mu(\alpha,A_{\pi})$.  If we replace just the  $i$'s in $\alpha$ by symbols chosen from $P_i$ in more
than $\forbmax(m,|P_i|,{\cal F}_i)$ ways then we get some $F\in {\cal
  F}_i$ with $F\prec A$, a contradiction. So 
  $\mu(\alpha,A_{\pi})\le
\prod_{i=0}^{t-1}\forbmax(m,|P_i|,{\cal F}_i)$. 
Combined with our bound on the number of different columns in $A_{\pi}$, we are done. \qed

\begin{remark}
The constant $c_{\ell, r}$ in Theorem \ref{allbuts} is
doubly exponential in $\ell$ in the proof above. One can reduce it
into $2^{c_r'\ell^2}$ for some constant $c'_r$ if we mimic the proof
of Theorem \ref{main} and use the Ramsey Theorem once. The details are
omitted here. 
  \end{remark}

When applying it to the partition $\{0,1\}\cup \{2\} \cup \{3\} \cup \cdots
\cup \{r\}$, we have the following theorem.

 \begin{thm}Let $r,\ell$ be given. There is a constant $c_{\ell, r}$
   so that the following statement holds.  Let ${\cal F}$ be a family
   of (0,1)-matrices. 
Then 
$$\forb(m,r,\left({\cal T}_{\ell}(r)\backslash {\cal
    T}_{\ell}(2)\right)
\cup \{T_{\ell+1}(0,1)\} \cup {\cal F})\le c_{\ell,r}\cdot \forbmax(m,s,{\cal F}).$$\label{01allbuts}\end{thm}
 This shows that, asymptotically at least, forbidding the configurations of 
 ${\cal T}_{\ell}(r)\backslash {\cal T}_{\ell}(2)\cup T_{\ell+1}(0,1)$ is much like restricting us to (0,1)-matrices.

\end{document}